\documentclass[aps,prd,nofootinbib,notitlepage,unsortedaddress]{revtex4-2}
\usepackage{enumerate}
\usepackage{dsfont}
\usepackage{amsfonts}
\usepackage{amssymb}
\usepackage[latin1]{inputenc}
\usepackage[T1]{fontenc}
\usepackage{graphicx}
\usepackage{mathtools}
\usepackage{bbm}
\usepackage{soul}

\usepackage[dvipsnames]{xcolor}

\usepackage[colorlinks=true]{hyperref}
\usepackage{amsthm}

\newcommand{\R}{\mathbb R}

\newcommand{\D}{\text{d}}

\newcommand{\dd}{\partial}
\newcommand{\db}{\bar\partial}

\newtheorem{prop}{Proposition}
\newtheorem{cor}[prop]{Corollary}
\newtheorem{lem}[prop]{Lemma}
\newtheorem{theor}[prop]{Theorem}

\begin{document}
\title{On the metrizability of $m$-Kropina spaces with closed null 1-form}

\author{Sjors Heefer}
\email{s.j.heefer@tue.nl}
\affiliation{Department of Mathematics and Computer Science, Eindhoven University of Technology, Eindhoven, The Netherlands}
\author{Christian Pfeifer}
\email{christian.pfeifer@zarm.uni-bremen.de}
\affiliation{ZARM, University of Bremen, 28359 Bremen, Germany.}

\author{Jorn van Voorthuizen}
\email{jornvanvoorthuizen@gmail.com}
\affiliation{Department of Mathematics and Computer Science, Eindhoven University of Technology, Eindhoven, The Netherlands}

\author{Andrea Fuster}
\email{a.fuster@tue.nl}
\affiliation{Department of Mathematics and Computer Science, Eindhoven University of Technology, Eindhoven, The Netherlands}

\begin{abstract}
We investigate the local metrizability of Finsler spaces with $m$-Kropina metric $F = \alpha^{1+m}\beta^{-m}$, where $\beta$ is a closed null 1-form. We show that such a space is of Berwald type if and only if the \mbox{(pseudo-)Riemannian} metric $\alpha$ and 1-form $\beta$ have a very specific form in certain coordinates. In particular, when the signature of $\alpha$ is Lorentzian, $\alpha$ belongs to a certain subclass of the Kundt class and $\beta$ generates the corresponding null congruence, and this generalizes in a natural way to arbitrary signature. We use this result to prove that the affine connection on such an $m$-Kropina space is locally metrizable by a \mbox{(pseudo-)Riemannian} metric if and only if the Ricci tensor constructed form the affine connection is symmetric. In particular we construct all counterexamples of this type to Szabo's metrization theorem, which has only been proven for positive definite Finsler metrics that are regular on all of the slit tangent bundle.
\end{abstract}

\maketitle
\tableofcontents


\section{Introduction}


The study of differences and similarities between positive definite Finsler geometry and indefinite Finsler geometry is still in its beginnings and far from complete \cite{Voicu2018,Fuster:2020upk,Javaloyes:2022hph}. The most prominent application of indefinite (to be precise Lorentzian) Finsler geometry is the one of Finsler spacetimes in classical and quantum gravitational physics \cite{Tavakol_1986,Voicu:2009wi,Pfeifer:2011xi,Lammerzahl:2018lhw,Pfeifer_2019,Hohmann_2019,Hohmann:2019sni,Lobo:2020qoa,Addazi:2021xuf, Kapsabelis:2022plf,Carvalho:2022sdz,Garcia-Parrado:2022ith,Zhu:2022blp}, which recently put \mbox{(pseudo-)Riemannian} geometry and its applications into the focus of interest \cite{Pfeifer:2011tk,Minguzzi:2014fxa,Gomez-Lobo:2016qik,Minguzzi2016,Javaloyes2018,Hohmann:2020mgs,Minculete_2021,Aazami:2022bib,Javaloyes:2022fmp}. Hence, a better understanding of the properties of indefinite Finsler geometry would be of great interest for physics as well as for mathematics. \\


Berwald spaces constitute an important class of Finsler spaces. They can be defined by the property that the canonical (Cartan) non-linear connection reduces to a linear connection on the tangent bundle \cite{Berwald1926}. It is natural to ask under what conditions this linear connection is (Riemann) \textit{metrizable}, in the sense that there exists a \mbox{(pseudo-)Riemannian} metric that has the given linear connection as its Levi-Civita connection. In positive definite Finsler geometry the answer to this question was given in 1988 by Szabo's well-known metrization theorem \cite{Szabo}, which guarantees that in this case the connection is \textit{always} metrizable. In the more general context, where the fundamental tensor is allowed to have arbitrary, not necessarily positive definite, signature, the situation is more complex. It only became clear very recently that Szabo's metrization theorem cannot be extended in general to arbitrary signatures~\cite{Fuster_2020}. In other words, there exist Finsler metrics of Berwald type (most examples being not positive definite and not smooth on the entire slit tangent bundle) for which the affine connection is not metrizable by a \mbox{(pseudo-)Riemannian} metric. 

It would be of great interest to know the precise conditions for metrizability in this more general context. As a first step in this direction, we investigate in this article the metrizability of a specific class of Finsler metrics, namely $m$-Kropina metrics with a closed null 1-form. The main result in this article, Theorem \ref{theorem:main}, states that the affine connection of such a space is metrizable if and only if the Ricci tensor constructed from the affine connection is symmetric, and gives a second equivalent characterization in terms of the local expression of the defining \mbox{(pseudo-)Riemannian} metric and 1-form. \\

$m$-Kropina metrics, also called generalized Kropina metrics, were introduced by Hashiguchi, Hojo and Matsumoto in \cite{Hashiguchi1973} as a generalization of the standard Kropina metric \cite{Kropina}. While the original Kropina metric has found a wide range of applications, $m$-Kropina metrics gained some popularity in the physics literature when it was discovered that they can be used to describe a modification of special relativity with local anisotropy \cite{Bogoslovsky,Bogoslovsky1973}, named very special relativity (VSR) \cite{Cohen:2006ky, Gibbons:2007iu} and later generalized to \textit{Very General Relativity} (VGR) \cite{Fuster:2018djw} or \textit{General Very Special Relativity} (GVSR) \cite{Kouretsis:2008ha} in order to account for spacetime curvature, leading to physical predictions from curved $m$-Kropina spacetime geodesics \cite{Elbistan:2020mca} and pp-waves\cite{Fuster:2015tua}. \\




The structure of this article is as follows. We start in section \ref{sec:FSFST} by recalling the basic notions of Finsler geometry that are relevant for our purpose and Szabo's metrization theorem for positive definite Berwald spaces. In section \ref{sec:VGR} we recall the definition of $m$-Kropina metrics and the precise necessary and sufficient condition under which they are of Berwald type (section \ref{sec:berwald_condition}). In fact we provide a new proof of this Berwald condition in apppendix \ref{app:BerwCondProof}. Subsequently in section \ref{sec:metrizability} we specialize to $m$-Kropina metrics constructed from a \mbox{(pseudo-)Riemannian} metric $\alpha$ and a 1-form $\beta$ that is null with respect to this metric and closed. We first prove Lemma \ref{lem:coordinates}, stating that such a space is of Berwald type if and only if $\alpha$ and $\beta$ have a very specific form in local coordinates. In particular, when the signature of $\alpha$ is Lorentzian, $\alpha$ belongs to a certain subclass of the Kundt class and $\beta$ generates the corresponding null congruence. This construction generalizes in a natural way to arbitrary signature. The coordinates introduced in this lemma allow us to find a simple expression for the linear connection coefficients and the skew-symmetric part of the affine Ricci tensor. We then prove our main result, Theorem \ref{theorem:main}, providing two equivalent necessary and sufficient conditions for metrizability: symmetry of the affine Ricci tensor and a local condition for the coordinate expressions of the \mbox{(pseudo-)Riemannian} metric $\alpha$. We end with a conclusion and discussion of the work in section \ref{sec:conc}.

\section{Finsler geometry}\label{sec:FSFST}
Finsler geometry is a natural extension of Riemannian geometry \cite{Finsler,Bao,Szilasi}. Given the philosophy that the length of a curve is obtained by integrating the norm of the tangent vector along the cuve, Finsler geometry provides the most general way of assigning, smoothly, a length to curves on a smooth manifold. While in Riemannian geometry the length of a tangent vector is given by a quadratic (metric-induced) norm, Finsler geometry relaxes this quadratic requirement. \\

First of all some remarks about notation are in order. Throughout this work we will usually work in local coordinates, i.e., given a smooth manifold $M$ we assume that some chart $\phi:U\subset M\to \R^n$ is provided, and we identify any $p\in U$ with its image $(x^1,\dots,x^n) = \phi(p)\in\R^n$. For $p\in U$ each $Y\in T_pM$ in the tangent space to $M$ at $p$ can be written as $Y = y^i\partial_i\big|_p$, where the tangent vectors $\partial_i \equiv \frac{\partial}{\partial x_i}$ furnish the chart-induced basis of $T_pM$. This provides natural local coordinates on the tangent bundle $TM$ via the chart
\begin{align}
\tilde\phi: \tilde U \to \R^n\times\R^n,\qquad \tilde U = \bigcup_{p\in U} \left\{p\right\}\times T_p M\subset TM,\qquad \tilde\phi(p,Y) = (\phi(p),y^1,\dots,y^n)\eqqcolon (x,y).
\end{align}
These local coordinates on $TM$ in turn provide a natural basis of its tangent spaces $T_{(x,y)}TM$, namely
\begin{align}
\bigg\{\partial_i\equiv \frac{\partial}{\partial x^i}, \bar{\partial}_i\equiv \frac{\partial}{\partial y^i}\bigg\}.
\end{align}

\subsection{Finsler spaces}

For our purposes, a Finsler space is triple $(M,\mathcal A,F)$, where $M$ is a smooth manifold, $\mathcal A$ is a conic subbundle\footnote{By conic subbundle we mean a non-empty open subset $\mathcal A\subset TM\setminus\{0\}$ such that for any $(x,y)\in\mathcal A$ it follows that $(x,\lambda y)\in\mathcal A$ for any $\lambda>0$.} of $TM\setminus\{0\}$ with non-empty fibers and $F$, the so-called Finsler metric, is a continuous map $F: TM\setminus\{0\}\to \mathbb R$, smooth on $\mathcal A$, that satisfies the following axioms:
\begin{itemize}
	\item $F$ is positively homogeneous of degree one with respect to $y$:
	\begin{align}
	F(x,\lambda y) =\lambda F(x, y)\,,\quad \forall \lambda>0\,;
	\end{align}
	\item The \textit{fundamental tensor}, with components $g_{ij} = \db_i\db_j \left(\frac{1}{2}F^2\right)$, is non-degenerate on $\mathcal A$.
\end{itemize}
In the positive definite setting (i.e. if $g_{ij}$ is assumed to be positive definite) one usually requires that $\mathcal A = TM\setminus\{0\}$. In the more general setting, however, this would exclude almost all interesting examples that have been studied in the literature. In fact there there is no consensus on a standard definition of Finsler space when the signature is indefinite (see e.g. \cite{Beem,Asanov,Pfeifer:2011tk, Pfeifer:2011xi, Javaloyes2014-1, Javaloyes2014-2}). A fundamental result essential for doing computations in Finsler geometry is Euler's theorem for homogeneous functions, which states that if a function $f$ is positively homogeneous of degree $r$, i.e., $f(\lambda y) =\lambda^r f(y)$ for all $\lambda>0$, then $y^i\frac{\dd f}{\dd y_i}(y) = r f(y)$. In particular, this implies the identity
\begin{align}
g_{ij}(x,y)y^i y^j = F(x,y)^2.
\end{align}
The coefficients of the \textit{Cartan non-linear connection}, the unique homogeneous (non-linear) connection on $TM$ that is smooth on $\mathcal A$, torsion-free and compatible with the Finsler metric, can be expressed as
\begin{align}
N^i_j(x,y) = \frac{1}{4}\bar{\partial}_j \bigg(g^{ik}\big(y^l\partial_l\bar{\partial}_k F^2 - \partial_k F^2\big)\bigg).
\end{align}
Torsion-freeness is the property that $\bar\partial_iN^k_j = \bar\partial_jN^k_i$, and metric-compatibility is the property that $\delta_k F^2 $ $\equiv \left(\partial_k-N^\ell_k\bar\partial_\ell\right)F^2 $ $=0$. Alternatively, metric compatibility can be defined by the property that $\nabla g_{ij} \equiv y^k\delta_k g_{ij} - N^k_i g_{kj} - N^k_j g_{ki} = 0$. For torsion-free homogeneous connections the latter definition of metric-compatibility is equivalent to the former. The curvature tensor, Finsler Ricci scalar and Finsler Ricci tensor 
of $(M,F)$ are defined, respectively, as
\begin{align}\label{eq:definition_curvatures}
R^i{}_{jk}(x,y) = -[\delta_j,\delta_k]^i =  \delta_j N^i_k(x,y)-\delta_k N^i_j(x,y),\qquad \text{Ric}(x,y) = R^i{}_{ij}(x,y)y^j,\qquad R_{ij}(x,y) = \frac{1}{2}\db_i \db_j\text{Ric}.
\end{align}

\subsection{Berwald spaces}
\label{sec:Berwald}

A Finsler space is said to be of Berwald type\footnote{See \cite{Szilasi2011} for an overview of the various equivalent characterizations of Berwald spaces and \cite{Pfeifer:2019tyy} for a more recent one in terms of a first order partial differential equation.} if the Cartan non-linear connection defines a linear connection on $TM$, or in other words, an affine connection on the base manifold, in the sense that the connection coefficients are of the form
\begin{align}
N^i_j(x,y) = \Gamma^i_{jk}(x)y^k
\end{align}
for a set of smooth functions $\Gamma^i_{jk}:M\to\R$. From the transformation behavior of $N^i_j$ it follows that the functions $\Gamma^i_{jk}$ have the correct transformation behavior to be the connection coefficients of a (torsion-free) affine connection on $M$. We will refer to this affine connection as the associated affine connection, or simply \textit{the} affine connection on the Berwald space. 
In addition to the curvature tensors defined in Eq. \eqref{eq:definition_curvatures} one may define additional curvature tensors for Berwald spaces, namely the ones associated to the uniquely defined affine connection. 

\begin{align}\label{eq:affine_curvatures}
\bar R_l{}^i{}_{jk}= 2\partial_{[j} \Gamma^i_{k]l} + 2\Gamma^i_{m[j}\Gamma^m_{k]l}, \qquad \bar R_{lk} = \bar R_l{}^i{}_{ik},
\end{align}
where we have employed the notation $T_{[ij]} = \frac{1}{2}\left(T_{ij}-T_{ji}\right)$ and $T_{(ij)} = \frac{1}{2}\left(T_{ij}+T_{ji}\right)$ for (anti-)symmetrization. We will refer to these as the \textit{affine curvature tensor} and the \textit{affine Ricci tensor}, respectively. We note that $\bar R_l{}^i{}_{jk}$ coincides (up to some reinterpretations) with the $hh$-curvature tensor of the Chern-Rund connection. A straightforward calculation reveals the following relation between the different curvature tensors
\begin{align}
\label{eq:symm_ricci}
R^j{}_{kl} = \bar R_i{}^j{}_{kl}(x)y^i, \qquad \text{Ric} = \bar R_{ij}(x)y^i y^j, \qquad R_{ij} = \frac{1}{2}\left(\bar R_{ij}(x) + \bar R_{ji}(x)\right).
\end{align}

It is appropriate to stress here that, although $R_{ij}$ and $\bar R_{ij}$ coincide in the positive definite setting and more generally whenever the Finsler function is defined on all of $\mathcal A = TM\setminus\{0\}$ \cite{Fuster_2020}, this is \textit{not} true in general, as $\bar R_{ij}$ need not be symmetric. As this distinction is essential for our results we end this section with a schematic overview of some important properties of the two Ricci tensors.\\

\noindent \textbf{Ricci tensors}
\begin{enumerate}
\item The \textit{Finsler Ricci Tensor} $ R_{ij}$ is constructed from the canonical non-linear connection associated to $F$, according to Eq. \eqref{eq:definition_curvatures}. The Finsler Ricci Tensor:
\begin{itemize}
\item Always exists;
\item Is symmetric, by definition;
\item Contains the same information as the \textit{Finsler Ricci scalar}. More precisely, Ric $=  R_{ij} y^i y^j$  and $R_{ij} = \frac{1}{2}\db_i \db_j\text{Ric}$.
\end{itemize}
\item The \textit{affine Ricci Tensor} $\bar R_{ij}$ is constructed from the affine connection associated to $F$ according to Eq. \eqref{eq:affine_curvatures}. The affine Ricci Tensor:
\begin{itemize}
\item Exists only for Berwald spaces, because otherwise there is no uniquely defined affine connection;
\item Coincides with the Ricci tensor\footnote{In principle one could list the Ricci tensors associated to the four well-known linear connections in Finsler geometry here as well, but we will not do so here. For Berwald spacetimes they are all identical anyway.} constructed from any of the four well-known linear connections associated to $F$ (Chern-Rund, Berwald, Cartan, Hashiguchi);
\item Is \textit{not} necessarily symmetric (except in the positive definite case). Its symmetrization coincides with the Finsler Ricci Tensor, see Eq. \eqref{eq:symm_ricci}.
\end{itemize}
\end{enumerate}

In this work we are primarily concerned with the \textit{affine Ricci tensor}, and in particular its property of being in general \textit{not} symmetric, as it can be used to characterize whether a given Finsler space is metrizable or not.


\subsection{Szabo's metrization theorem}

Given a Finsler space of Berwald type, the Cartan non-linear connection defines a linear connection on $TM$, by definition. Hence the natural question arises whether there exists a \mbox{(pseudo-)Riemannian} metric (desirably of the same signature) that has this connection  as its Levi-Civita connection. Simply put, is every Berwald space metrizable? For positive definite Finsler spaces defined on all of $TM\setminus\{0\}$ the answer is affirmative, as proven by Szabo \cite{Szabo}.

\begin{theor}[Szabo's metrization theorem]
Any positive definite Berwald space is metrizable by a Riemannian metric.
\end{theor}

The proof of this theorem relies on averaging procedures \cite{CrampinAveraging} for which it is essential that the Finsler function $F$ is defined everywhere on $TM\setminus\{0\}$. In the case of Finsler \textit{spacetimes}, however, the domain where $F$ is defined is typically only a conic subset of $TM\setminus\{0\}$ and hence the classical proof does not extend to this case. It was indeed shown in \cite{Fuster_2020} that Szabo's metrization theorem is in general not valid for Finsler spacetimes. The culprit behind all counterexamples known to the authors is the fact that the affine Ricci tensor is in general not symmetric. Clearly the property that the affine Ricci tensor be symmetric is a necessary condition for metrizability. We will see (Theorem \ref{theorem:main}) that for $m$-kropina spacetimes with closed 1-form this is in fact also a sufficient condition, at least locally.

\section{m-Kropina metrics}\label{sec:VGR}
An $m$-Kropina space (sometimes called generalized Kropina space) is a Finsler space of $(\alpha,\beta)$-type with a Finsler function of the form $F = \alpha^{1+m}\beta^{-m}$, where $\alpha = \sqrt{a_{ij}(x)y^iy^j}$ is constructed from a \mbox{(pseudo-)Riemannian} metric $a = a_{ij}\D x^i\D x^j$, $\beta = b_iy^i$ is constructed from a 1-form $b = b_i\D x^i$ and $m$ is a real parameter. By a slight abuse of terminology one also refers to $\alpha$ and $\beta$ simply as the \mbox{(pseudo-)Riemannian} metric and the 1-form, respectively. We also introduce the notation $b^2 \equiv ||b||^2 = a_{ij}b^ib^j$ for the squared norm of $\beta$ with respect to $\alpha$. Throughout the remainder of this article all indices are raised and lowered with $a_{ij}$. \\
In the physics literature spacetimes with metric of $m$-Kropina type have been dubbed \textit{Very General Relativity} (VGR) spacetimes \cite{Fuster:2018djw} or \textit{General Very Special Relativity} (GVSR) spacetimes \cite{Kouretsis:2008ha}, introduced as generalizations of \textit{Very Special Relativity} (VSR) \cite{Cohen:2006ky,Gibbons:2007iu}, which appears in the limiting case where $\alpha$ is flat. In the latter case the corresponding $m$-Kropina metric is often referred to as the Bogoslovsky line element. When $m=1$ the $m$-Kropina metric reduces to the standard Kropina metric \cite{Kropina} $F = \alpha^2/\beta$.

\subsection{The Berwald condition}\label{sec:berwald_condition}
The Berwald condition for m-Kropina spaces $F = \alpha^{1+m}\beta^{-m}$ formulated by Matsumoto in\footnote{The result is proven only for \textit{non-null} 1-forms $\beta$ in Theorem 6.3.2.3 on page 904 of \cite{handbook_Finsler_vol2}, but as long as the dimension of the manifold is greater than 2 the proof is still completely valid also for null 1-forms, as the other results around Theorem 6.3.2.3 clearly show.} \cite{handbook_Finsler_vol2} states that such a space is of Berwald type if and only if there exists a vector field $f^i$ on $M$ such that
\begin{align}\label{eq:VGR_Berwald_new}
\nabla_j b_i = m (f_k b^k)a_{ij} + b_i f_j  - m f_i b_j.
\end{align}
Here and throughout the remainder or the article $\nabla$ denotes the Levi-Civita connection corresponding to the \mbox{(pseudo-)Riemannian} metric $\alpha$. In the special case that $\beta$ is a closed and hence locally exact 1-form, any $f_k$ satisfying this condition can always be written as $f_k = cb_k$ for some function\footnote{Note that our $c$ is related to $C(x)$ in \cite{Fuster:2018djw} by $C(x) = (1+m)c/2$. Also note that our power $m$ is related to the power $n$ in \cite{Fuster:2018djw} by $n = -2m/(1+m)$.} $c$ on the base manifold and the condition reduces to the simpler one obtained in \cite{Fuster:2018djw}, namely
\begin{align}\label{eq:VGR_Berwald_old}
\nabla_j b_i = c \left[ mb^2 a_{ij} + (1-m)b_i b_j\right],
\end{align}
To see this, assume that Matsumoto's Berwald condition \eqref{eq:VGR_Berwald_new} holds. We have $(\D b)(\partial_i,\partial_j)= \partial_i b_j -\partial_j b_i = \nabla_i b_j - \nabla_j b_i = (1+m)(f_i b_j - f_j b_i)$, so if $b_i$ is locally exact then this expression vanishes and hence $f_i b_j = f_j b_i$ must hold for all $i,j$, which is only possible if $f_i$ is proportional to $b_i$ (this can be checked easily at any given point in $M$ by choosing coordinates in which $b_i$ has only one non-vanishing component at that point). In other words, $f_k = cb_k$. In this case \eqref{eq:VGR_Berwald_new} reduces to \eqref{eq:VGR_Berwald_old}. \\
Note that the opposite holds (trivially) as well: the latter condition implies that $\beta$ is locally exact. The fact that \eqref{eq:VGR_Berwald_new} and \eqref{eq:VGR_Berwald_old} do not agree for 1-forms $\beta$ that are not closed has recently caused some confusion and this seems like a good opportunity to resolve this issue. It turns out the reason for the discrepancy is that the contribution of the anti-symmetric part of the covariant derivative of $\beta$ was overlooked in the proof given in \cite{Fuster:2018djw}. Indeed in appendix \ref{app:BerwCondProof} we reproduce the argument from \cite{Fuster:2018djw} taking the anti-symmetric part into account, and we show that the resulting Berwald condition coincides with \eqref{eq:VGR_Berwald_new}, as expected. Thus, we want to stress here again that \eqref{eq:VGR_Berwald_new} is the correct Berwald condition in general, whereas \eqref{eq:VGR_Berwald_old} only applies to the case in which the 1-form $\beta$ is closed.\\

As also proved in \cite{handbook_Finsler_vol2}, whenever the condition \eqref{eq:VGR_Berwald_new} is satisfied, the affine connection coefficients of the Berwald space can be expressed in terms of the Christoffel symbols ${}^\alpha\Gamma^k_{ij}$ for the Levi-Civita connection corresponding to $\alpha$ as
\begin{align}
\Gamma^\ell_{ij} = {}^\alpha\Gamma^\ell_{ij} + m a^{\ell k}\left(a_{ij}f_k - a_{jk}f_i - a_{ki}f_j\right).
\end{align}
When the 1-form $\beta$ is closed, and we write $f_k = cb_k$ as before, this reduces to
\begin{align}\label{eq:affine_conn}
\Gamma^\ell_{ij} = {}^\alpha\Gamma^\ell_{ij} + m c \left(a_{ij}b^\ell - \delta^\ell_j b_i - \delta^\ell_i b_j\right),
\end{align}
which agrees with the result obtained in \cite{Fuster:2018djw}.

\subsection{Metrizability of $m$-Kropina spaces with closed null 1-form}
\label{sec:metrizability}

From here onwards we will focus on $m$-Kropina metrics with closed null 1-form and we will assume that $n=\dim M>2$. In other words, we will assume that $\D b = 0$ and $b^2 = a_{ij}b^ib^j=0$. This will allow us to deduce the exact conditions for local metrizability. As a remark we point out that the case $m=1$ is excluded by our definition of Finsler space, as can be seen from the expression of the determinant of the fundamental tensor,
\begin{align}
\frac{\det g}{\det a}  = (1+m)^3\alpha^{8m}|\beta|^{-2(1+4m)}\left(m b^2\alpha^2 + (1-m)\beta^2\right),
\end{align}
which vanishes identically when $m=1$ and $b^2=0$.\\

The following lemma extends a result from \cite{Gomez-Lobo:2016qik}. We will use the convention that indices $a,b,c,\dots$ run from $3$ to $n$, whereas indices $i,j,k,\dots$ run from $1$ to $n$.

\begin{lem}\label{lem:coordinates}
$F$ is Berwald if and only if around each $p\in M$ there exist local coordinates $(u,v,x^3,\dots,x^n)$ such that\footnote{By the product notation $\D u\D v$ we mean the symmetrized tensor product $\D u \D v \equiv (1/2)(\D u\otimes \D v + \D v \otimes \D u)$ and similarly in other instances.}
\begin{align}
{\normalfont a = -2\D u\D v + H(u,v,x)\D u^2 + W_a(u,x)\D u\D x^a + h_{ab}(u,x)\D x^a\D x^b,\quad b = \D u},
\end{align}
with $h$ some \mbox{(pseudo-)Riemannian} metric. In this case the metric satisfies the Berwald condition $\eqref{eq:VGR_Berwald_old}$ with
\begin{align}
c =- \frac{\partial_v H}{2(1-m)}.
\end{align}
\end{lem}
\begin{proof}
First, we may pick coordinates $(v,x^2,\dots,x^n)$ around $p$ adapted to $b$ in the sense that $b =\partial_v$, i.e. $b^i = \delta^i_1$. At this point the metric has the general form $a = a_{ij}\D x^i\otimes\D x^j$. (Abusing notation a little bit, $b$ sometimes denotes the 1-form and sometimes the vector field uniquely corresponding to it via the isomorphism induced by $a$. It should be clear from context which is meant.). The null character of $b$ manifests as the fact that $a_{11}=a_{vv}=0$ in these coordinates. Because $b$ is closed and hence locally exact, we may write, locally, $b= \D u$ for some function $u(v,x^2,\dots,x^n)$.  Equivalently, $b_i = \partial_i u$. Note also that $\partial_i u = b_i = a_{ij}b^j = a_{ij}\delta^j_1 = a_{i 1}$. Since $a_{11}=0$ it follows that $\partial_v u=\partial_1 u=0$. As $b\neq 0$ by assumption, there must be some $i\geq 2$ such that $\partial_i u = a_{i1}\neq 0$ in a neighborhood of $p$. Order the coordinates $x^2,\dots,x^n$ such that this is true for $i=2$, i.e. assume without loss of generality that $a_{21}\neq 0$. Next define the map
\begin{align}
x = (v,x^2,\dots,x^n)\mapsto \tilde x = (v,u(x^2,\dots,x^n),x^3,\dots,x^n).
\end{align}
Its Jacobian matrix and its inverse are given by 
\begin{align}
J^i{}_j = \frac{\partial\tilde x^i}{\partial x^j} = \begin{pmatrix}
1 & 0 & 0 & 0&\hdots & 0 \\
0 & a_{21} & a_{31} & a_{41} & \hdots & a_{n1} \\
0 & 0 & 1 & 0 & \hdots & 0 \\
0 & 0 & 0 & 1 & \hdots & 0\\
\vdots & \vdots & \vdots & \vdots & \ddots & \vdots \\
0 & 0 & 0 & 0 & \hdots & 1
\end{pmatrix},\qquad (J^{-1})^i{}_j = \frac{\partial x^i}{\partial \tilde x^j} &= \begin{pmatrix}
1 & 0 & 0 & 0&\hdots & 0 \\
0 & 1/a_{21} & -a_{31}/a_{21} & -a_{41}/a_{21} & \hdots & a_{n1}/a_{21} \\
0 & 0 & 1 & 0 & \hdots & 0 \\
0 & 0 & 0 & 1 & \hdots & 0\\
\vdots & \vdots & \vdots & \vdots & \ddots & \vdots \\
0 & 0 & 0 & 0 & \hdots & 1
\end{pmatrix}
\end{align}
and since $\det J = a_{21}\neq 0$ this matrix is invertible, so $x\mapsto \tilde x$ is a local diffeomorphism at $p$. It remains to find the form of the metric in the new coordinates. We have
\begin{align}
\tilde a_{ij} = \frac{\partial x^k}{\partial \tilde x^i} \frac{\partial x^\ell}{\partial \tilde x^j}a_{k\ell},\qquad \text{i.e.} \qquad \tilde a = J^{-1T}aJ^{-1}.
\end{align}
Therefore
\begin{align}
\tilde a_{11} &= (J^{-1}{}^T)_1{}^i a_{ij}(J^{-1}{})^j{}_1 = a_{11} = 0,\\
\tilde a_{12} &= (J^{-1}{}^T)_1{}^i a_{ij}(J^{-1}{})^j{}_2 = 1 \\
\tilde a_{1b} &= (J^{-1}{}^T)_1{}^i a_{ij}(J^{-1}{})^j{}_b = a_{12}(-a_{b1}/a_{21}) + a_{1b} = 0,\qquad b = 3,\dots, n.
\end{align}
This shows that $a = \tilde a_{ij}\D \tilde x^i\D \tilde x^j = 2\D u\D v + H\D u^2 + W_b\D u\D x^b + h_{bc}\D x^b\D x^c$ for certain functions $H$, $W_a$, $h_{ab}$, and hence after a redefinition $v \to -v$ we may write the metric in the form
\begin{align}
a = -2\D u\D v + H\D u^2 + W_b\D u\D x^b + h_{bc}\D x^b\D x^c.
\end{align}
It follows from the easily checked fact that $\det h = -\det a \neq 0$ that $h_{ab}$ is itself a \mbox{(pseudo-)Riemannian} metric of dimension $n-2$.\\
Our arguments thus far are independent of whether the $m$-Kropina space is of Berwald type or not. All we have used is that the \mbox{(pseudo-)Riemannian} metric $a$ admits a 1-form that is null and closed. We will prove next that the $m$-Kropina space is Berwald if and only if the functions $W_a$ and $h_{ab}$ do not depend on coordinate $v$. To this end we employ the Berwald condition \eqref{eq:VGR_Berwald_new}. In fact, since the 1-form is assumed to be closed we may use the simpler version, Eq. \eqref{eq:VGR_Berwald_old}. And since the 1-form is null ($b^2=0$) as well, this condition reduces to
\begin{align}\label{1}
\nabla_i b_j = c(1-m)b_ib_j.
\end{align}
The $m$-Kropina space is Berwald if and only if there exists a function $c$ on $M$ such that this condition is satisfied. On the other hand, computing $\nabla_i b_j$ explicitly in the new coordinates, using the fact that $b_i = \delta^u_i$ and $g^{ui} = -\delta^i_v$ and $g_{iv} = 0$, yields
\begin{align}\label{2}
\nabla_ib_j  = -\frac{1}{2}\frac{\partial a_{ij}}{\partial v}.
\end{align}
Combining equations \eqref{1} and \eqref{2}, using again that $b_i = \delta^u_i$, yields $c(1-m)\delta^u_i\delta^u_j = -\partial_va_{ij}/2$, or equivalently,
\begin{align}
c =- \frac{\partial_v H}{2(1-m)}\qquad \&\qquad \partial_v W_a = \partial_vh_{ab} = 0.
\end{align}
From this it follows that $F$ is Berwald if and only if $\partial_v W_a = \partial_vh_{ab} = 0$, and that $c$ is in that case given by the desired expression, completing the proof.
\end{proof}

From here onwards we will assume our space is Berwald. Substituting the form of $c$ into Eq. \eqref{eq:affine_conn} and using that $b_i = \delta^u_ i$ and consequently $b^\ell = a^{\ell k}b_k =a^{\ell k}\delta^u_k = a^{\ell u} = -\delta^\ell_v$, we obtain the following.

\begin{cor}\label{cor:connection_coeff}
In the coordinates $(u,v,x^3,\dots,x^n)$, the affine connection coefficients can be expressed in terms of the Levi-Civita Christoffel symbols ${}^\alpha\Gamma^k_{ij}$ of the \mbox{(pseudo-)Riemannian} metric $\alpha$ as
\begin{align}
\Gamma^k_{ij}  = {}^\alpha\Gamma^k_{ij}  + \Delta\Gamma^k_{ij}  \equiv {}^\alpha\Gamma^k_{ij}  + \frac{m}{2(1-m)}\partial_v H \left(a_{ij}\delta^k_v + \delta^k_j \delta^u_ i + \delta^k_i \delta^u_ j\right)
\end{align}
\end{cor}

We can use the preceding results to analyze the (deviation from the) symmetry of the affine Ricci tensor, which has a very simple expression in these coordinates, as the following result shows.

\begin{lem}\label{lem:ricci}
In the coordinates $(u,v,x^3,\dots,x^n)$, the skew-symmetric part of the affine Ricci tensor is given by
\begin{align}
\bar R_{[ij]} = -\frac{mn}{4(1-m)}(\delta^u_i\partial_j\partial_vH-\delta^u_j\partial_i\partial_vH)
\end{align}
\end{lem}
\begin{proof}
From the definition \eqref{eq:affine_curvatures} of the affine Ricci tensor of a Berwald space it follows that its skew-symmetric part can be written as
\begin{align}
\bar R_{[ij]} \equiv \frac{1}{2}\left(\bar R_{ij} - \bar R_{ji}\right) = \partial_{[k}\Gamma^k_{j]i} - \partial_{[k}\Gamma^k_{i]j}.
\end{align}
We now use the expression for the connection coefficients found in Corollary \ref{cor:connection_coeff}. Note that 
\begin{align}
\Delta\Gamma^k_{kj} = \frac{m}{2(1-m)}\partial_v H\left(a_{vj} + \delta^u_ j+ n \delta^u_j\right) = \frac{mn}{2(1-m)}\partial_v H \delta^u_j.
\end{align}

Substituting this in the skew-symmetric part of the affine Ricci tensor we obtain
\begin{align}
\bar R_{[ij]} &= \partial_{[k}\Gamma^k_{j]i} - \partial_{[k}\Gamma^k_{i]j} = \partial_{[k}\Delta\Gamma^k_{j]i} - \partial_{[k}\Delta\Gamma^k_{i]j} = \frac{1}{2}\left(-\partial_{j}\Delta\Gamma^k_{ki} + \partial_{i}\Delta\Gamma^k_{kj}\right) \\
&= \frac{mn}{4(1-m)}\left(\delta^u_j\partial_{i}\partial_v H-\delta^u_i\partial_{j}\partial_v H  \right),
\end{align}
where we have used the fact that the Ricci tensor corresponding to $\alpha$ is symmetric.
\end{proof}

Let us now prove our main result.

\begin{theor}\label{theorem:main}
Let $(M,F=\alpha^{1+m}\beta^{-m})$ be an $m$-Kropina space of Berwald type with closed null 1-form $\beta$. The following are equivalent:
\begin{enumerate}[(i)]
\item The affine connection is locally metrizable by a \mbox{(pseudo-)Riemannian} metric
\item The affine Ricci tensor is symmetric, $\bar R_{ij} = \bar R_{ji}$
\item There exist local coordinates $(u,v,x^3,\dots,x^n)$ 
such that
\begin{align}
{\normalfont a = -2\D u\D v + \left[\tilde H(u,x) + \phi(u)v\right]\D u^2 + W_a(u,x)\D u\D x^a + h_{ab}(u,x)\D x^a\D x^b,\quad b = \D u},
\end{align}
with $h$ some \mbox{(pseudo-)Riemannian} metric of dimension $n-2$.
\end{enumerate}

In this case the affine connection is metrizable, in the chart corresponding to the coordinates $(u,v,x^3,\dots x^n)$, by the following \mbox{(pseudo-)Riemannian} metric:
\begin{align}\label{eq:metrizing_metric}
{\normalfont \tilde a = e^{\frac{m}{1-m}\int^u \phi(\tilde u)\D \tilde u } a}
\end{align}
\end{theor}

Before we present the proof, we want to point out two things. First, we note that if $\phi = 0$ then $\tilde a = a$, i.e. the affine connection is metrizable by the defining \mbox{(pseudo-)Riemannian} metric $\alpha$. This was to be expected, since in that case the 1-form $\beta$ is parallel with respect to $\alpha$. It is a well-known result that any $(\alpha,\beta)$-metric for which $\beta$ is parallel with respect to $\alpha$, is of Berwald type, and that its affine connection coicides with the Levi-Civita connection of $\alpha$. Second, since $\tilde a$ is conformally equivalent to $a$, the two metrics have identical causal structure and moreover their null geodesics coincide (as unparameterized curves). This implies that the null geodesics of any $F$ satisfying any (and hence all) of the equivalent conditions of Theorem \ref{theorem:main} coincide with the null geodesics of the defining \mbox{(pseudo-)Riemannian} metric $\alpha$.

\begin{proof}
(i) trivially implies (ii). For (ii)$\Rightarrow$(iii) we use the preferred coordinates introduced in the lemma above. By Lemma \ref{lem:ricci}, the only non-vanishing skew symmetric components of the affine Ricci tensor are
\begin{align}
\bar R_{[uj]} = -\frac{mn}{4(1-m)}\partial_j\partial_vH,\qquad j = 2,\dots, n.
\end{align}
Note that the fact that there's an index $u$ on the LHS and an index $v$ on the RHS is not a typo. The anti-symmetric part of the $uj$ component of the Ricci tensor, are determined by the $vj$-derivative of $H$. By assumption the Ricci tensor is symmetric. The $uv$ component therefore yields  $\partial_v^2 H = 0$ and the remaining components yield $\partial_v\partial_{a} H=0$, $a = 3,\dots n$. %
%
In other words, $H$ must be linear in $v$ and the corresponding linear coefficient can depend only on the coordinate $u$.  That is,

\begin{align}
g = -2\D u\D v + \left[\tilde H(u,x) + \phi(u)v\right]\D u^2 + W_a(u,x)\D u\D x^a + h_{ab}(u,x)\D x^a\D x^b.
\end{align}

This proves $(ii)\Rightarrow(iii)$. For the last implication $(iii)\Rightarrow (i)$, recall from Corollary \ref{cor:connection_coeff} that the affine connection coefficients can be expressed as
\begin{align}
\Gamma^k_{ij}= {}^\alpha\Gamma^k_{ij}  + \Delta\Gamma^k_{ij}  \equiv {}^\alpha\Gamma^k_{ij}  + \frac{m}{2(1-m)}\phi(u) \left(a_{ij}\delta^k_v + \delta^k_j \delta^u_ i + \delta^k_i \delta^u_ j\right)
\end{align}

On the other hand, an elementary calculation shows that the Levi-Civita Christoffel symbols of a \mbox{(pseudo-)Riemannian} metric $\tilde a = e^{\psi(u)}a$ can be expressed in terms of the original Christoffel symbols as
\begin{align}
{}^{\tilde a}\Gamma^\ell_{ij} = {}^\alpha\Gamma^k_{ij} +\frac{1}{2}\psi'(u) \left(a_{ij}\delta^\ell_v + \delta^\ell_j \delta^u_ i + \delta^\ell_i \delta^u_ j\right).
\end{align}
Hence, since $\psi'(u) = -2mc = \frac{m}{1-m}\phi(u)$ for the \mbox{(pseudo-)Riemannian} metric $\tilde a$ indicated in the theorem, it follows that the connection coefficients of $\tilde a$ coincide with the affine connection coefficients of our $m$-Kropina metric. This completes the proof of the theorem.
\end{proof}

Theorem \ref{theorem:main} provides necessary and sufficient conditions for an m-Kropina space with closed null 1-form to be locally metrizable. In the next section we apply our results to an explicit example from the physics literature.

\subsection{An explicit example: Finsler VSI spacetimes}

In this section we apply our results to the Finsler VSI spacetimes presented in \cite{Fuster:2018djw}, with the 4-dimensional Finsler metric
\begin{align}
F = \left(-2\D u \D v +\left[\tilde H(u,x,y) + \phi(u,x,y)v\right]\D u^2 + W_1(u,x,y)\D u\D x + W_2(u,x,y)\D u\D y +\D x^2 + \D y^2\right)^{(1+m)/2}\left(\D u\right)^{-m},
\end{align}
By Lemma 2, this spacetime is of Berwald type. It is in general not metrizable since the corresponding affine Ricci tensor is not symmetric. By Theorem \ref{theorem:main}, the exact condition for metrizability in this case is that $\partial_x\phi = \partial_y\phi = 0$. The case $\phi=0$ provides a Finsler version of the gyratonic pp-wave metric \cite{gyraton,Maluf2018}, which according to Theorem \ref{theorem:main} is metrizable by the Lorentzian gyratonic pp-wave metric.\\

A simple non-trivial locally metrizable example is provided by the case where $\tilde H(u,x,y) = 0$, $\phi(u,x,y)=u$ and $W_a(u,x,y) =0$. This leads to the Finsler metric
\begin{align}
F = \left(-2\D u \D v + u\,v\,\D u^2 + \D x^2 + \D y^2\right)^{(1+m)/2}\left(\D u\right)^{-m},
\end{align}
which has an affine connection given by following non-vanishing affine connection coefficients:
\begin{align}
\Gamma^u_{uu} = \frac{1+m}{2(1-m)}u, \quad \Gamma^v_{uu} = -\frac{1-m-u^2}{2(1-m)}v, \quad \Gamma^v_{uv} = -\frac{u}{2}, \quad \Gamma^v_{xx} = \Gamma^v_{yy} = \Gamma^x_{ux} = \Gamma^y_{uy} = -\frac{-m}{2(1-m)}u.
\end{align}

As indicated by Eq. \eqref{eq:metrizing_metric} in Theorem \ref{theorem:main} this connection is metrizable by the Lorentzian metric
\begin{align}
\tilde g = e^{\frac{m u^2}{2(1-m)} }\left(-2\D u \D v + u\,v\,\D u^2 + \D x^2 + \D y^2\right).
\end{align}



\section{Discussion}\label{sec:conc}

Recent developments around the non-metrizability of Berwald spaces of indefinite (in particular, Lorentzian) signature contrast the well-known metrizability theorem by Szabo for \textit{positive definite} Berwald spaces. These findings inspired us to investigate the question of metrizability for $m$-Kropina Finsler metrics constructed from a \mbox{(pseudo-)Riemannian} metric and a closed null 1-form, in this article. Our main result, Theorem \ref{theorem:main}, gives a necessary and sufficient condition for local metrizability, namely that the \textit{affine Ricci tensor} -- the Ricci tensor constructed from the affine connection, not to be confused with the more commonly discussed Finsler Ricci tensor -- must be symmetric.


Also, in the coordinates introduced in Lemma \ref{lem:coordinates}, any Berwald $m$-Kropina metric attains a pretty simple form. It can then be seen at a glance whether a given geometry is locally metrizable or not. Moreover, in the metrizable case our theorem gives the explicit form of a (non-unique) \mbox{(pseudo-)Riemannian} metric that `metrizes' the affine connection, in terms of those coordinates.\\

The question of metrizability is not only a natural one from the mathematical point of view, but it is also of interest in the realm of physics. In particular, in the field of Finsler gravity,  which asserts that the spacetime geometry of our physical universe might be Finslerian. One of its postulates is that physical objects and light rays moving only under the influence of gravity, follow Finslerian geodesics through spacetime. If the Finsler metric on spacetime were metrizable, this would imply that these trajectories reduce to the geodesics of a \mbox{(pseudo-)Riemannian} metric, precisely as is the case in Einstein gravity\footnote{It is not the case that the whole theory can be reduced to \mbox{(pseudo-)Riemannian} geometry, however. In order accurately describe, for instance, the causal character of the geodesics, or the flow of proper time attributed to observers, one would still need the Finsler metric.}. Apart from obvious mathematical implications, it would be interesting to investigate the conceptual and physical consequences of this as well.\\

It would obviously be of great interest to have a generalization of Theorem \ref{theorem:main} to arbitrary Finsler spaces of Berwald type. To this effect, we note that, curiously, all examples of non-metrizable Berwald spaces currently available in the literature, as well as all of the additional examples known privately to the authors, have an affine Ricci tensor that is not symmetric. Together with the results obtained in this article in the specific case of $m$-Kropina metrics, this leads us to hypothesize that perhaps a Berwald space is metrizable by a \mbox{(pseudo-)Riemannian} metric if and only if its affine Ricci tensor is symmetric. In fact, some general results about Riemann-metrizability of arbitrary symmetric affine connections are known \cite{Tanaka:2011we,Tamassy1996,Schmidt1973}. An affine connection is metrizable if and only if the holonomy group is a subgroup of the generalized \mbox{(pseudo-)orthogonal} group \cite{Schmidt1973}. Hence, a future project is to investigate the structure of the holonomy group of the affine connection corresponding to a Berwald space and how it relates to the geometry-defining Finsler function.

\begin{acknowledgments}
 C.P. was funded by the Deutsche Forschungsgemeinschaft (DFG, German Research Foundation) - Project Number 420243324 and acknowledges support from cluster of excellence Quantum Frontiers funded by the Deutsche Forschungsgemeinschaft (DFG, German Research Foundation) under Germany's Excellence Strategy - EXC-2123 QuantumFrontiers - 390837967. All of us would like to acknowledge networking support by the COST Action CA18108, supported by COST (European Cooperation in Science and Technology).
\end{acknowledgments}
	
\newpage
\appendix

\section{Proof of the Berwald condition for m-Kropina metrics}\label{app:BerwCondProof}
Here we provide a proof of the Berwald condition \eqref{eq:VGR_Berwald_new} for m-Kropina spaces $F = \alpha^{1+m}\beta^{-m}$, which also serves as extension for the proof presented in \cite{Fuster:2018djw}, where it was overlooked that the 1-form $\beta$ need not be closed. The derivations in this section have been performed with help of the xAct extension of Mathematica \cite{xact}.

The Finsler metric $L$ for $m$-Kropina spaces is given by $F = (a_{ij}(x)y^i y^j)^{\frac{1+m}{2}} (b_k(x)y^k)^{-m}$. Using the decomposition
\begin{align}
	\nabla_i b_j =  A_{ij}  + S_{ij}\,,
\end{align}
where $A_{ij} = A_{[ij]}(x)$ is the anti-symmetric and $S_{ij} = S_{(ij)}(x)$ is the symmetric part of the covariant derivative, we find a geodesic spray $G^j = N^j_i y^i$ of the form
\begin{align}
	G^k 
	&= {}^\alpha\Gamma^k{}_{ij}(x)y^i y^j \\
	&+ y^{i} A^k{}_{i}\frac{\alpha^2 m  }{\beta (m+1)}
	- b^{k} \frac{\alpha^2 m  (\beta (m+1)  y^{i} y^{j} S_{ij}  -2 m  \alpha^2 b^{i} y^{j} A_{ij})}{2 \beta (m+1) \left(\beta^2 (m-1)-b^2 \alpha^2 m\right)}
	+  y^{k} \frac{m (\beta (m+1) y^{i} y^{j} S_{ij} - 2 m  \alpha^2 b^{i} y^{j} A_{ij})}{(m+1) \left(\beta^2 (m-1)-b^2 \alpha^2 m\right)}\,.\label{eq:GeodFinsBerwald}
\end{align}
Indices are raised and lowered with the components of the \mbox{(pseudo-)Riemannian} metric defining $\alpha$. In order to be of Berwald type the components $G^k$ need to be quadratic functions of $y$. This is the case, since for a Berwald space $N^j_i(x,y) = \Gamma^j_{ik}(x)y^k$ and so $G^j(x,y) = N^j_i(x,y) y^i = \Gamma^j_{ik}(x)y^k y^i$.

To reach this goal, the first term in \eqref{eq:GeodFinsBerwald} must either cancel with one of the other terms appearing or, the contraction $y^{b} A^k{}_{b}$ must lead to a term proportional to $\beta$. Hence, the free index on  $y^{b} A^k{}_{b}$ must either be on $y^k$, $b^k$ or $Z^k=Z^k(x)$, where $Z^k$ are the components of another vector field $Z$ on $M$, in the following way
\begin{align}
	y^{j} A^k{}_{j} = \beta Z^k + y^k T + b^k U_i y^i\,,
\end{align}
for $T = T(x)$ being a function and $U_i=U_i(x)$ being the components of a 1-form on $M$. These are the only possible terms, since by construction, $y^{b} A^k{}_{b}$ is a linear function in $y$, and so the RHS must be as well. Factoring the linear dependence in $y$ on both sides of the equation leads to
\begin{align}
	A^k{}_{j} = b_j Z^k + T \delta^k_j + b^k U_j\,,
\end{align}
which then implies by the anti-symmetry $A_{ij} = A_{[ij]}$
\begin{align}\label{eq:asymcdb}
	A_{ij} = b_i (U_j - Z_j) - b_j (U_i - Z_i)\,.
\end{align}
Defining $f_j=  (U_j - Z_j)$, we see that a necessary condition for a m-Kropina space to be Berwald is that the anti-symmetric part of the covariant derivative of the 1-form $\beta$ is determined by $b_i$ and an additional 1-form with components $f_j$. Using
\begin{align}
	\nabla_i b_j = (m+1) (b_i f_j - b_j f_i)  + S_{ij}\,,
\end{align}
where the factor $(m+1)$ was added in front of the anti-symmetric part to display the following expressions more compactly. For the geodesic spray one finds
\begin{align}\label{eq:GKropina}
	G^k 
	= {}^\alpha\Gamma^k{}_{ij}(x)y^i y^j 
	&+ m \alpha^2 f^k\\
	&+ m \alpha^2 b^k \frac{2(m \alpha^2 b^i - (m-1)\beta y^i)f_i - S_{ij}y^i y^j}{2((m-1)\beta^2 - m b^2 \alpha^2)} 
	+ m y^k \frac{2 m \alpha^2 (b^2 y^i - \beta b^i) f_i + \beta S_{ij}y^i y^j}{((m-1)\beta^2 - m b^2 \alpha^2)}\,.\label{eq:GKropina2}
\end{align}
The use of the derived expression for the anti-symmetric part of the covariant derivative \eqref{eq:asymcdb} ensures that the second term in the geodesic spray above is quadratic in $y$. To achieve this for the third term for the case $m\neq 1$ let us investigate the structure of the $y$-dependence of this term. It is of the type
\begin{align}
	B(y,y)\frac{C(y,y)-S(y,y)}{D(y,y)}\,,
\end{align}
where each term $X(y,y)=X_{ij}y^i y^j, X=B,C,D,S,P$, denotes a quadratic polynomial in $y$ and $B(y,y)=\alpha^2$. In order for this function to be quadratic in $y$ it must satisfy
\begin{align}
	B(y,y)\frac{C(y,y)-S(y,y)}{D(y,y)} = P(y,y) \quad \Leftrightarrow\quad C(y,y)-S(y,y) = \frac{P(y,y)D(y,y)}{B(y,y)}\,,
\end{align}
for some second order polynomial $P(y,y)$. Since the left hand side is a second order polynomial in y, the right hand side must be. By construction $B(y,y)$ is an irreducible quadratic polynomial in $y$. As long as $m\neq 1$, $D\neq h(x) B(y,y)$. Thus, $P(y,y)$ must satisfy $P(y,y) = h(x) B(y,y)$, for a solution of the equation to exist. Hence, the fraction in the first term of line \eqref{eq:GKropina2} must be proportional to an arbitrary function $h=h(x)$ on $M$. This yields to the equation
\begin{align}
	2(m \alpha^2 b^i - (m-1)\beta y^i)f_i - S_{ij}y^i y^j = h (2((m-1)\beta^2 - m b^2 \alpha^2))\,.
\end{align}
Taking two derivatives w.r.t. $y$ we find
\begin{align}
	S_{ij} = m a_{ij} (h b^2 + 2 b^k f_k) - (m-1) (b_i f_j + b_j f_i + h b_i b_j)\,.
\end{align}
Redefining $f_i$ as $f_i = \frac{1}{2} (\tilde f_i -  b_i h)$ and combining all expressions for the covariant derivative of $\beta$ finally gives the desired expression \eqref{eq:VGR_Berwald_new}
\begin{align}
	\nabla_i b_j = m (\tilde f^k b_k) a_{ij} + b_j \tilde f_i - m b_i \tilde f_j\,.
\end{align}
One can easily check that this condition on $b_j$ leads to a geodesic spray 
\begin{align}
	G^k = {}^\alpha\Gamma^k{}_{ij}(x)y^i y^j + \frac{m}{2} \alpha^2 \tilde f^k - m y^k y^i \tilde f_i\,,
\end{align}
which indeed is quadratic, and so the $m$-Kropina space subject to the condition \eqref{eq:VGR_Berwald_new} is indeed Berwald.\\

For $m=1$ and $b^2 \neq 0$ the first term in line \eqref{eq:GKropina2} is quadratic in $y$ for any tensor components $S_{ij}$ and we must investigate the second term of that line, which becomes
\begin{align}
	- m y^k \frac{2 \alpha^2 (b^2 y^i - \beta b^i) f_i + \beta S_{ij}y^i y^j}{(b^2 \alpha^2)}\,,
\end{align}
and can only be quadratic in $y$ if and only if
\begin{align}
	 \frac{\beta S_{ij}y^i y^j}{(b^2 \alpha^2)} = Q_i y^i
\end{align}
for some $1$-form on $M$ with components $Q_i=Q_i(x)$. The only way to achieve this is if $S_{ij} = q a_{ij}$ for some function $q=q(x)$ on $M$, which then must satisfy
\begin{align}
	 b_i q = Q_i b^2 \Rightarrow Q_i\sim b_i,\ q \sim b^2 \,.
\end{align}
Thus, for $m=1$
\begin{align}
	S_{ij} \sim a_{ij}  b^2\,.
\end{align}
For $m=1$ and $b^2 = 0$ the determinant of the metric $g$ vanishes globally, and hence this situation does not define a Finsler space or spacetime.

%


\bibliographystyle{utphys}
\bibliography{VGRmetrizability}

\end{document}